\documentclass{article}
\usepackage[cp1251]{inputenc}
  \usepackage[english]{babel}
  \usepackage{amsfonts,amssymb,amsmath,amstext,amsbsy,amsopn,amscd,amsthm,graphicx,euscript}
  \usepackage{graphics}
  \usepackage{amsthm}

\newtheorem{thm}{Theorem}
\newtheorem{lem}{Lemma}
\newtheorem{re}{Remark}
\newtheorem{ex}{Example}

\newtheorem{crl}{Corollary}

\newcounter{tdfn}
\newcounter{trk}
\def\R{{\mathbb R}}

 \def\Z{{\mathbb Z}}
 \def\0{{\mathbbf 0}}
 \def\1{{\mathbbf 1}}

 \def\Z{{\mathbf Z}}
 \def\G{{\cal G}}

 \newcommand{\skcrossl}{\raisebox{-0.25\height}{\includegraphics[width=0.5cm]{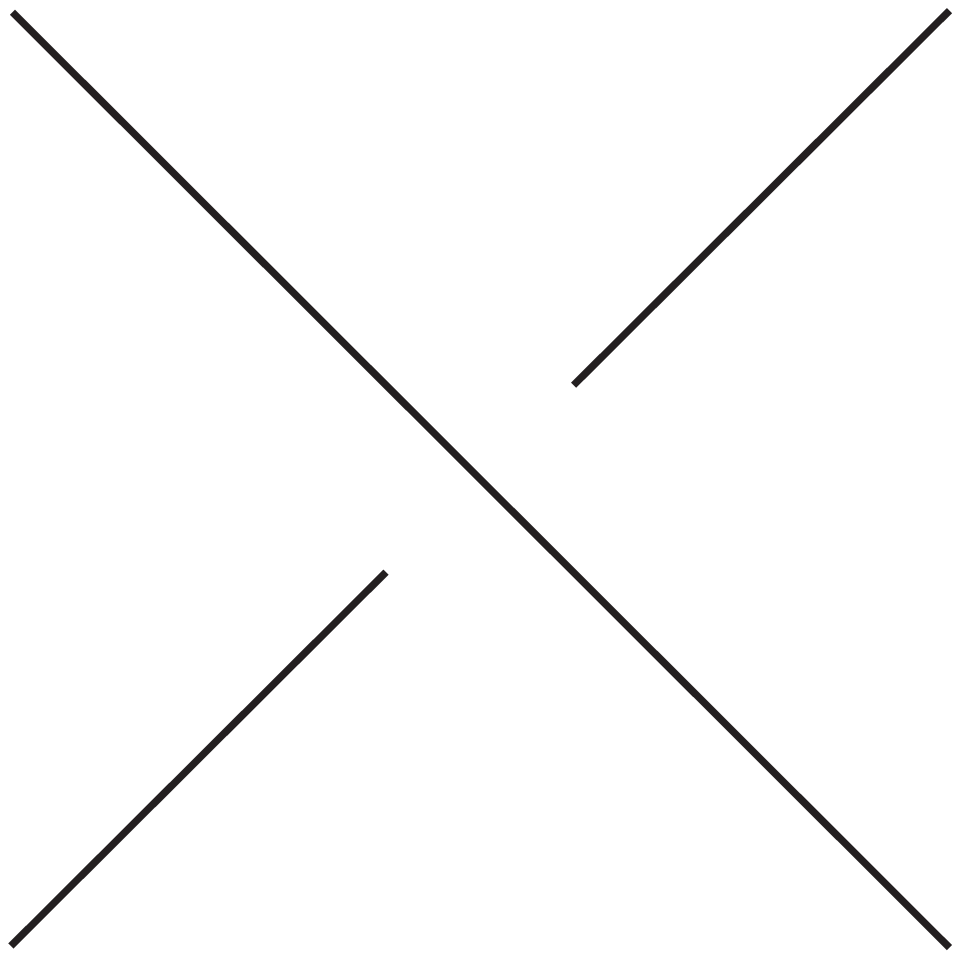}}}
 \newcommand{\skcrossr}{\raisebox{-0.25\height}{\includegraphics[width=0.5cm]{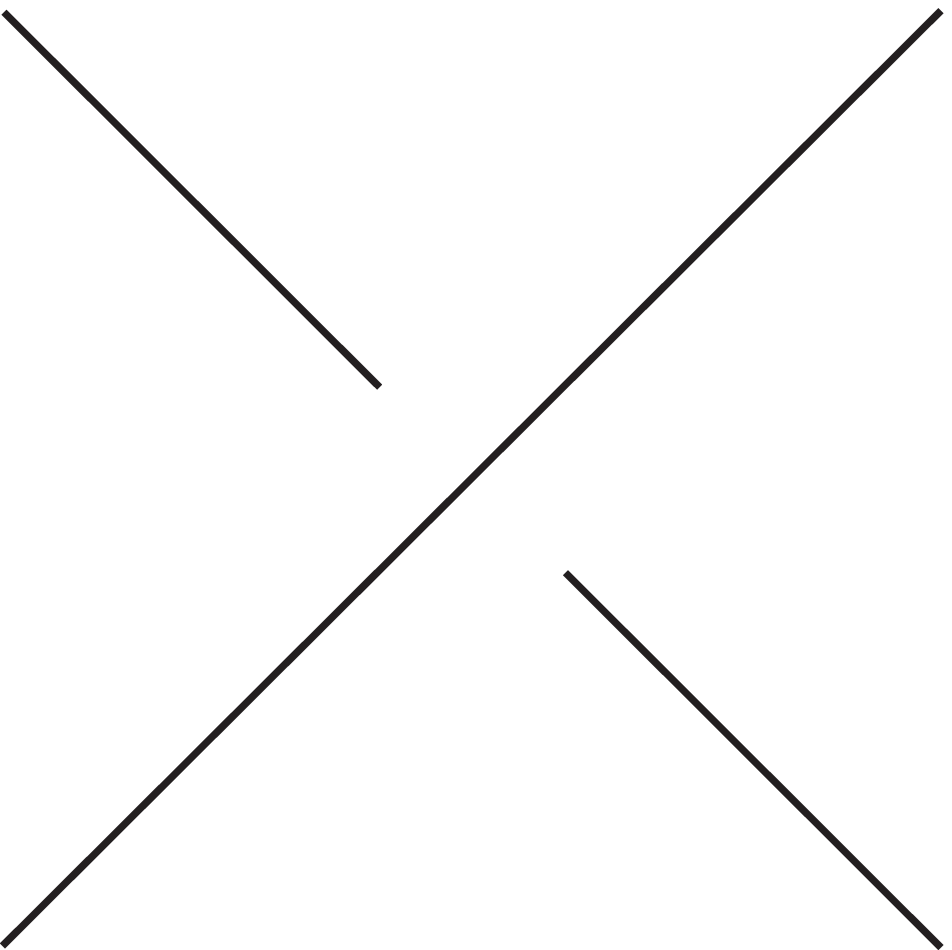}}}

\title{Virtual Crossing Numbers for Virtual Knots}

\author{Vassily Olegovich Manturov\footnote{Peoples' Friendship University of Russia, vomanturov at yandex.ru}\footnote
{The author is partially supported by  grants of RFBR 10-01-00748-a,
RF President NSh -- 3224.2010.1, Ministry of Education and Science
of the Russian Federation 14.740.11.0794 and the Analytic
Departmental Task Program ``Development of the High School
Scientific Potential''}}

\date{}

\begin{document}

\maketitle

\begin{flushright}
To the memory of my father \\ Oleg Vassilievich Manturov \\
(July,3,1936 - July, 23,2011).

\end{flushright}

\begin{abstract} The aim of the present paper is to prove that the minimal
 number of virtual crossings for some families of virtual knots grows
quadratically with respect to the minimal number of classical
crossings.

All previously known estimates for virtual crossing number
(\cite{Af,DK,ST} etc.) were principally no more than linear in the
number of classical crossings (or, what is the same, in the number
of edges of a virtual knot diagram) and no virtual knot was found
with virtual crossing number greater than the classical crossing
number.
\end{abstract}



\section{Introduction}

The main idea of the present paper is to use the {\em parity
arguments}: if there is a smart way to distinguish between {\em
even} and {\em odd} crossings of a virtual knot so that they behave
nicely under Reidemeister moves then there is a way to reduce some
problems about {\em virtual knots} to analogous problems about {\em
their diagrams (representatives)}.

Thus, we have to find a certain family of four-valent graph for
which the crossing number (minimal number of {\em additional
crossings} (prototypes of virtual crossings) for an immersion in
$\R^{2}$) is quadratic with respect to the number of {\em vertices}
(prototypes of classical crossings).

The study of parity has been first undertaken in \cite{Sbornik1},
see also \cite{Sbornik2,MyNewBook} where functorial mappings from
virtual knots to virtual knots were constructed, minimality theorems
were proved, and many virtual knot invariants were refined. In the
paper \cite{Projection}, by using parity, I constructed a
diagrammatic projection mapping from virtual knots to classical
knots.

In the case of graphs, such families having quadratic growth for the
number of additional crossings with respect to the number of the
crossings themselves are quite well known to graph theorists: even
for trivalent graphs the generic crossing number grows quadratically
with respect to the number of vertices, see, e.g., \cite{PSSz}.

{\bf Notational remark.} For graphs, we shall use the standard
terminology: the number of vertices $v$, and the crossing number
$cr$, the latter referring to the minimal number of additional
crossings for generic immersions, see ahead. For virtual knots, we
shall use the notation: $vi(K)$ and $cl(K)$ for minimal virtual
crossing number and minimal classical crossing number over all
diagrams of a given knot.

\subsection{Acknowledgements}

I wish to express my special gratitude to D.P.Ilyutko, O.V.Manturov,
L.H.Kauffman, S.Jablan, and A.V.Yudin for fruitful discussions.

I am especially grateful to Yu.S.Makarychev for fruitful
consultations concerning graph theoretical questions.

\section{Virtual Knots and Crossing Numbers}

A {\em virtual diagram} is a four-valent graph on the plane where
each crossing is either {\em classical} (in this case one pair of
opposite edges are marked as an overcrossing pair, and the other
pair is marked as an undercrossing pair; the undercrossing pair is
drawn by means of a broken line) or {\em virtual} (virtual crossings
are encircled).

Another way of looking at a virtual diagram is as follows. We say
that a four-valent graph is {\em framed} if at every crossing of it,
the  four (half)edges incident to this crossings are split into two
sets of (formally) opposite half-edges. An immersion of a
four-valent graph in $\R^{2}$ is  {\em generic} if all points having
more than one preimage are intersection points of exactly two edges
at their interior points. Then a virtual diagram is a generic
immersion of a four-valent framed graph with all images of graph
vertices endowed with classical crossing structure and all
intersection points between edges marked as virtual crossings. Those
points with more than one preimage are called {\em crossing points}.

A virtual {\em link} is an equivalence class of virtual diagrams
modulo the {\em detour move} and the classical Reidemeister moves.
Classical Reidemeister moves deal with classical crossings only;
they refer to a domain of the plane with no virtual crossings
inside. The detour move is the move which can be viewed as a
transformation of the immersion outside the images of classical
crossings: it takes an arc containing virtual crossings (and,
possibly, self-crossings of an edge with itself) only and replaces
it with an arc having the same ends but drawn in another way (all
new crossings are to be virtual).

A virtual knot is a one-component virtual link. In this paper we
deal with virtual knots only, however, the argument can be easily
modified for the case of links.

The {\em classical (resp., virtual) crossing number} $cl(K)$ (resp.,
$vi(K)$) of a virtual knot $K$ is the minimum of the numbers of
classical (resp., virtual) crossings over all diagrams of $K$.

Classical crossing numbers of virtual knots were studied for a long
time, see, e.g. \cite{MyNewBook}, and references therein.

For estimates of virtual crossing numbers for virtual knots see
\cite{Af,DK,ST}.

In the last years, some attempts to compare the classical and
virtual crossing numbers were undertaken, e.g., Satoh and Tomiyama
\cite{ST} proved that for any two positive numbers $m<n$ there is a
virtual knot $K$ with minimal virtual crossing number $vi(K)=m$ and
minimal classical crossing number $cl(K)=n$.

However, no results were found in the opposite direction: for all
known virtual knots the number of classical crossings was greater
than or equal to the number of virtual crossings (see tables due to
J.Green \cite{Green}).

In the present paper, we disprove this conjecture by reducing the
problem {\em from knots to graphs}: we take some family of graphs
for which $cr$ grows quadratically with respect to the number of
vertices, transform them into four-valent graphs (which can
correspond to diagrams of virtual knots with classical vertices
corresponding to crossings), turn these graph into a good shape
(irreducibly odd, see ahead) by some transformations which increase
the complexity a little, and then use the fact that for irreducibly
odd graphs the  crossing number is equal to the virtual crossing
number of the underlying knots.

\subsection{$4$-Graphs and Free Knots}

Now, let us change the point of view to virtual knots and consider
some much simpler objects.

By a {\em $4$-graph} we mean either a split sum of several
$1$-complexes each of which is either a regular finite $4$-graph
(loops and multiple edges are admitted) or is homeomorphic to a
circle. By a {\em vertex} of a $4$-graph we mean a vertex of some of
its graph components. By an {\em edge} we mean either an edge of
some of its graph components or a {\em whole} circle component. The
latter are called {\em circular} edges.

All edges which are not circular are considered as equivalence
classes of {\em half-edges}. We say that a $4$-graph is {\em framed}
if for each vertex of it, the four half-edges incident to this
vertex are split into two pairs of (formally opposite) half-edges.

By a {\em unicursal component} of a framed $4$-graph we mean either
some of its circular components or an equivalence class of edges of
some graph component, where the equivalence is defined as follows.
Two edges $a,b$ are {\em equivalent} if there exists a chain of
edges $a=a_{1},\dots, a_{n}=b$ for which each two adjacent edges
$a_{i},a_{i+1}$ have two half edges $a'_{i},a'_{i+1}$ which are
opposite at some vertex.

A framed $4$-graph is {\em oriented} if all its circular components
are oriented, and all edges of its graph components are oriented in
such a way that at each vertex, for each pair of opposite edges, one
of them is incoming, and the other one is emanating.

By a {\em free link} we mean an equivalence class of framed
$4$-graphs by the following equivalences (three Reidemeister moves):

The first Reidemeister move is an addition/removal of a loop, see
Fig.\ref{1r}, left.

\begin{figure}
\centering\includegraphics[width=100pt]{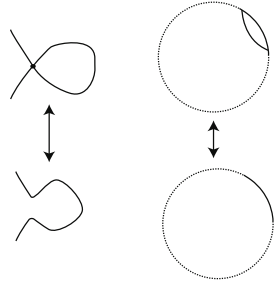}
\caption{Addition/removal of a loop on a graph and on a chord
diagram} \label{1r}
\end{figure}

The second Reidemeister move adds/removes a bigon formed by a pair
of edges which are adjacent in two edges, see Fig. \ref{2r},top.

\begin{figure}
\centering\includegraphics[width=150pt]{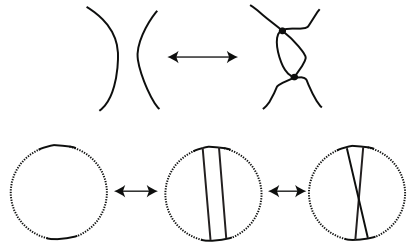} \caption{The second
Reidemeister move and two chord diagram versions of it} \label{2r}
\end{figure}

Note that the second Reidemeister move adding two vertices does not
impose any conditions on the edges it is applied to.

The third Reidemeister move is shown in Fig.\ref{3r},top.

\begin{figure}
\centering\includegraphics[width=150pt]{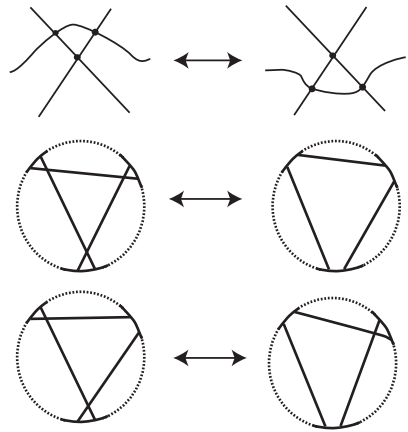} \caption{The third
Reidemeister move and its chord diagram versions} \label{3r}
\end{figure}

Note that these transformations may turn a circular component into a
{\em unicursal component of a framed $4$-graph}.

The orientation of framed $4$-graphs naturally leads to the notion
of {\em oriented free links}.

One can easily see that the number of {\em unicursal components} of
a framed $4$-graph does not change under the Reidemeister moves. So,
one can speak about the number of {\em unicursal components} of a
free link. By a {\em component} of a link we mean an unicursal
component unless specified otherwise.

A {\em free knot} is a $1$-component free link.

Clearly, free knots and free links are equivalence classes of
virtual knots and virtual links by the following two equivalences:
the {\em crossing switch} $\skcrossr\longleftrightarrow \skcrossl$
and the {\em virtualisation move}; the latter move flanks a
classical crossing by two virtual crossings, as shown in Fig.
\ref{virtua}.

\begin{figure}
\centering\includegraphics[width=200pt]{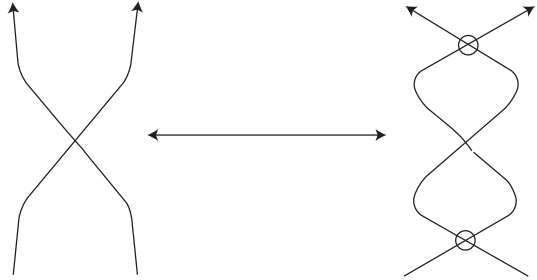}
\caption{Virtualisation} \label{virtua}
\end{figure}

The meaning of the first move is that we forget which branch of a
knot is going {\em over} in a classical crossing (the other branch
goes under); the meaning of the second move is that we allow to flip
the cyclic clockwise (half)edge order at a crossing from $1,2,3,4$
to $1,4,3,2$.

\subsection{Chord diagrams}

A {\em chord diagram} is a finite cubic graph consisting of an
 cycle (the {\em core}) passing through all vertices and a collection of
non-oriented edges connecting vertices. We also admit {\em the empty
chord diagram} which is just a circle (in this case the circle is
the core). A chord diagram is {\em oriented} if its core is
oriented.

Chord diagrams are in one-to-one correspondence with framed
$4$-graphs having one unicursal component. We associate with the
empty chord diagram the circle (the framed $4$-graph with one
component and no vertices) ; with any other chord diagram $D$ we
associate the framed $4$-graph as follows. We take the $1$-complex
obtained from the chord diagram $C$ as follows. Take the core $Co$
of the chord diagram $C$ and identify those points connected by
chords; we get a $4$-graph; for this $4$-graph we say that two
(half)-edges are opposite if they come from two (half)-edges
approaching the same chord end on $Co$. Certainly, {\em oriented}
chord diagrams are in a bijective correspondence with {\em oriented}
framed $4$-graphs with one unicursal component.

For chord diagrams, the Reidemeister moves look as shown in
Fig.\ref{1r},right,\ref{2r},bottom,\ref{3r}, centre and bottom.

We say that two chords $A,B$ of a chord diagram $C$ are {\em linked}
if the two ends of the chord $B$ lie in distinct connected component
of the complement to the endpoints of $A$ in the core circle of the
Gauss diagram, and {\em unlinked} otherwise. For any chord $A$ we
say that $A$ is unlinked with itself. We say that a chord $A$ of a
chord diagram  is {\em even} if it is linked with evenly many
chords, and {\em odd} otherwise.

Analogously, for a framed $4$-graph with one unicursal component we
say that a crossing is {\em even} (resp., {\em odd}) iff the
corresponding chord is {\em even} (resp., {\em odd}).

A chord diagram (resp., framed $4$-graph with one unicursal
component) is {\em odd} if all chords of it are odd.
 We say that an odd four-valent framed graph with one unicursal
 component is {\em irreducibly odd} if no second decreasing Reidemeister move
can be applied to it. At the level of chord diagram this means that
there are no two chords $A,B$ such that one end of $A$ is adjacent
to one end of $B$ on the core circle, and the other end of $A$ is
adjacent to the other end of $B$.

The importance of this notion is the following: the oddness of a
framed $4$-graph means that neither decreasing first Reidemeister
move or a third Reidemeister move can be applied to it. The
irreducible oddness also requires that no second Reidemeister move
would be applicable. So, irreducibly odd framed $4$-graphs can be
operated on only by those Reidemeister moves which increase the
number of crossings.

An irreducibly odd diagram is given in Fig. \ref{irredodd}.

\begin{figure}
\centering\includegraphics[width=100pt]{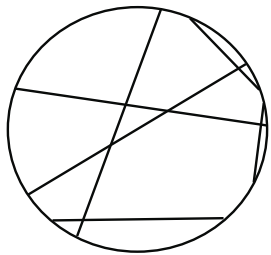} \caption{An
irreducibly odd diagram} \label{irredodd}
\end{figure}

As we shall see further, odd chords play a crucial role in the study
of free knots.

\subsection{The bracket}

In the present section we shall introduce a simple invariant of free
knots which allows one  {\em to reduce many problems about knots to
problems about their representatives}.

We shall start with the notion of smoothing. Let $\Gamma$ be a
framed $4$-graph. By {\em smoothing} of $\Gamma$ at $v$ we mean any
of the two framed $4$-graphs obtained by removing $v$ and repasting
the edges as $a-b$, $c-d$ or as $a-d,b-c$, see Fig. \ref{smooth}.
The rest of the graph (together with all framings at vertices except
$v$) remains unchanged. We may then consider further smoothings of
$\Gamma$ at {\em several} vertices.

\begin{figure}
\centering\includegraphics[width=150pt]{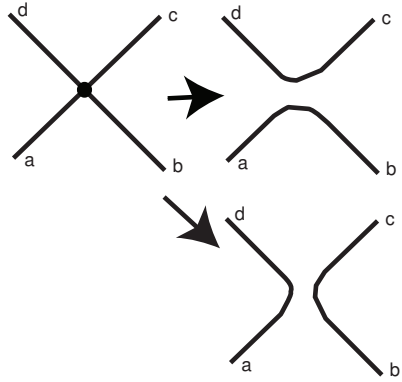} \caption{Two
smoothings of a vertex of for a framed graph} \label{smooth}
\end{figure}

Note that this operation may lead to circular connected components
of the graph.

Let $\G$ be the set of equivalence classes of all four-valent framed
graphs modulo the second Reidemeister move. Consider the formal
$\Z_{2}$-linear space generated by all classes from $\G$.

Now, for a given framed $4$-graph $\Gamma$, consider the following
sum

\begin{equation}
[\Gamma]=\sum_{s\;even.,1\; comp} \Gamma_{s},\label{eqlo}
\end{equation}
which is taken over all smoothings in all {\em even} vertices, and
only those summands are taken into account where $\Gamma_{s}$ has
one unicursal component.

Thus, if  $\Gamma$ has $k$ even vertices, then $[\Gamma]$ contains
at most $2^{k}$ summands, and if all vertices of $\Gamma$ are odd,
then we shall have exactly one summand, the graph $\Gamma$ itself.
Consider  $[\Gamma]$ as an element of $\Z_{2}\G$. In this case it is
evident that if all vertices of $\Gamma$ are even then
$[\Gamma]=[\Gamma_0]$: by construction, all summands in the
definition of $[\Gamma]$ are equal to $[\Gamma_0]$, it can be easily
checked that the number of such summands is odd.

Now, we are ready to formulate the main result of the present
section:

\begin{thm}[\cite{Sbornik1}]
If $\Gamma$ and $\Gamma'$ represent the same free knot then in
$\Z_{2}\G$ the following equality holds:
$[\Gamma]=[\Gamma']$.\label{mainthm}
\end{thm}

Theorem \ref{mainthm} yields the following
\begin{crl}
Let  $\Gamma$ be an irreducibly odd framed 4-graph with one
unicursal component. Then any representative $\Gamma'$ of the free
knot $K_{\Gamma}$, generated by $\Gamma$, has a smoothing  $\tilde
\Gamma$ equivalent to $\Gamma$ as a framed $4$-graph. In particular,
$\Gamma$ is a minimal representative of the free knot $K_{\Gamma}$
with respect to the number of vertices.\label{sld}
\end{crl}

Indeed, if we look at an irreducibly odd graph $\Gamma$, we see that
$[K_{\Gamma}]=\Gamma$. In the left hand side of this equality,
$K_\Gamma$ means a free knot, i.e., an equivalence class of a
$4$-graph
 $\Gamma$ modulo the three Reidemeister moves. In the right hand
 side, we have just the graph $\Gamma$ modulo the second
 Reidemeister moves.

In fact, classification of elements from $\G$ is very easy.

Two graphs are equivalent whenever their two minimal representatives
coincide.

So, in this case one can say that {\em the bracket} takes {\em
dynamical objects} (framed $4$-graphs modulo Reidemeister moves) to
{\em statical objects} (framed $4$-graphs modulo just the second
Reidemeister moves, or just their minimal representatives).

In this way, in \cite{Sbornik1} I proved that free knots are
generally not invertible: this was done by means of finding a good
non-invertible representative for free links and some other
orientation-sensitive parity arguments.


\subsection{Crossing number for graphs}

Given a graph $\Gamma$; analogously to the case of four-valent
graphs, by a {\em generic} immersion of $\Gamma$ in $\R^{2}$ we mean
an immersion $\Gamma\to \R^{2}$ such that

\begin{enumerate}

\item the number of points with more than one preimage is finite;

\item tach such point has exactly two preimages;

\item these two preimages are interior points of edges of the graph,
and the intersection of the images of edges at such a point is
transverse.

\end{enumerate}

By {\em crossing number} $cr(K)$ of a graph $\Gamma$ we mean the
minimal number of crossing points over all generic immersions
$\Gamma\to \R^{2}$.

When we deal with framed $4$-graphs, we restrict ourselves for such
immersions for which at the image of every vertex the images of any
two formally opposite edges turn out to be opposite on the plane.

\begin{ex}
Consider the only $4$-graph with one vertex $A$ and two edges $p,q$
connecting $A$ to $A$. There are two possible framings for this
graph; one of these framings (where one half edge of $p$ is formally
opposite to the other half edge of $p$) leads to a framed $4$-graph
with two unicursal components. Such a graph is certainly non-planar,
and its crossing number is equal to one, see Fig. \ref{smplgrh}. The
other framing (where a half-edge of the edge $p$ is opposite to a
half-edge of the edge $q$) is planar, so, for that framing the
crossing number is $0$.
\end{ex}

\begin{figure}
\centering\includegraphics[width=120pt]{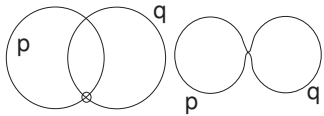} \caption{A
$4$-graph with two framings} \label{smplgrh}
\end{figure}

Now, let us present some examples of graphs where the crossing
number grows quadratically. Let $p$ be a prime number; consider the
chord diagram with $(p-3)/2$ chords obtained as follows: take a
standard circle $x^{2}+y^{2}=1$ on the plane, take all residue
classes modulo $p$ except $0,p-1,1$, and put the residue class on
the standard (core) circle as follows: the vertex corresponding to
the residue class $r$ will be located at $(\cos \frac{ 2\pi r}{p},
\sin \frac{2\pi r}{p})$. Now, every crossing $r$ is coupled with the
crossing $s$ where $rs \cong 1 \; mod\; p$.

It is known that for such graphs for $p\to \infty$ the crossing
number grows quadratically in $p$.

Other examples of families of trivalent graphs with quadratic growth
can be constructed by using {\em expander family}; for more about
expanders, see, e.g., \cite{PSSz}. The idea is as follows: for a
graph $\Gamma$ and a set $V$ of vertices of it, we define the
neighbourhood $N(V)$ to be the set of vertices of $\Gamma$ not from
$V$ which are connected to at least one vertex from $V$ by an edge.
It is natural to study the ratio $\frac{|N(V)|}{|V|}$. A family
$F_{n}$ of graphs is called an $\varepsilon$-expander family for
some positive constant $\varepsilon$ if this ratio exceeds
$\varepsilon$ for all graphs $F_{n}$ for sufficiently large $n$ and
for all sets $V_{n}$ of vertices smaller than the half of all
vertices of $F_{n}$.

\section{The Main Theorem}

We are now ready to state and to prove our main result.

\begin{thm}
For some infinite set of positive integers $i$, there is a family
$V_{i}$ of virtual knots such that the virtual crossing number of
$V_{i}$ grows quadratically with respect to the classical crossing
number of $V_{i}$ as $i$ tends to the infinity.
\end{thm}

The proof of this theorem relies upon the following lemmas.

\begin{lem}
Let $K$ be a framed $4$-graph. Let $K'$ be a graph obtained from $K$
by smoothings at some vertices. Then $cr(K')\le cr(K)$.
\label{lm1}
\end{lem}

\begin{lem} Let $L_{n}$ be a family of trivalent graphs such that the
crossing number $cr(L_{n})$ grows quadratically with respect to the
number of vertices $v(L_{n})$ as $n$ tends to the infinity. Then
there are two families of framed $4$-graphs $\Gamma'_{n}$
$\Gamma_{n}$ such that

\begin{enumerate}

\item $\Gamma_{n}$ are all irreducibly odd;

\item The number of vertices of $\Gamma_{n}$ does not exceed
$3$ times the number of vertices of $L_{n}$.

\item $\Gamma'_{n}$ is obtained from $\Gamma_{n}$ by smoothing of
some vertices; both $\Gamma_{n}$ and $\Gamma'_{n}$ are graphs with
one unicursal component;

\item $L_{n}$ is a subgraph of $\Gamma'_{n}$ obtained by removing some edges.

\end{enumerate}
\label{lm2}
\end{lem}

\begin{proof}[Proof of the Main Theorem]
Let us take a family of trivalent graphs $L_{n}$ with quadratical
growths of the crossing number. Denote their numbers of vertices by
$v_{n}$ and denote their crossing numbers by $cr_{n}$.

Apply Lemma \ref{lm2}. Consider the families of graphs $\Gamma_{n}$
and $\Gamma'_{n}$. Consider an arbitrary immersion of $\Gamma_{n}$
in $\R^{2}$. Endow all vertices of this immersion with any classical
crossing structure; denote the obtained virtual knot by $K_{n}$.

We claim that the classical crossing number $cl(K_{n})$ of the knot
$K_{n}$ grows linearly with respect to $cr_{n}$, whence the virtual
crossing number $vi(K_{n})$ grows quadratically with respect to
$cr_{n}$. The first claim follows from the construction: the number
of classical crossings of $K_{n}$ does not exceed tree times the
number of vertices of $L_{n}$, so, the minimal classical crossing
number over all diagrams representing the knot given by $K_{n}$ can
be only smaller.

Now, consider $vi(K_{n})$. Let $L$ be a diagram of the knot
represented by $K_{n}$.  So, if we consider the framed $4$-graphs
corresponding to diagrams $L$ and $K_{n}$, they will represent the
same free knot. By definition, $K_{n}$ corresponds to the framed
$4$-graph $\Gamma_{n}$. Denote the framed $4$-graph corresponding to
$L$ by $\Delta$. We see that $\Delta$ represent the same free knot
as $\Gamma_{n}$. Now, apply Theorem \ref{sld} to the free knot
generated by $\Gamma_{n}$. By construction, it is irreducibly odd.
Thus, we see that $\Gamma_{n}$ can be obtained from $\Delta$ by
means of a smoothing at some vertices.

So, by Lemma \ref{lm1}, the (virtual) crossing number of $\Delta$ is
bounded from below by the (virtual) crossing number $vi(K_{n})$ of
$K_{n}$.  By definition, $vi(K_{n})$, in turn,  is bounded from
below by the crossing number of $\Gamma_{n}$. By Lemma \ref{lm1},
the latter is estimated from below by $\Gamma'_{n}$. Finally,
$cr(\Gamma'(n))\ge cr(L_{n})$ because $L_{n}$ is a subgraph of
$\Gamma_{n}$, and $cr(L_{n})$ grows quadratically with respect to
the number of vertices of $L_{n}$.

This completes the proof of the Main Theorem.

\end{proof}

Now let us prove auxiliary Lemmas \ref{lm1} and \ref{lm2}.

\begin{proof}[Proof of Lemma \ref{lm1}]
Indeed, consider an immersion of $K'$ in $\R^{2}$ preserving the
framing and realising the crossing number $cr(K')$. Now, take those
vertices of $K$ where the smoothing $K'\to K$ takes place and
perform this smoothing just on the plane.
\end{proof}

\begin{proof}[Proof of Lemma \ref{lm2}]
Let $L_{n}$ be a connected trivalent graph. Obviously, $n$ is even;
let us couple the vertices of $L_{n}$ arbitrarily and connect
coupled vertices by edges. We get a four-valent graph. We shall
denote it by $\Gamma'_{n}$; to complete the construction of
$\Gamma'_{n}$, we have to find a framing for it in order to get a
diagram of a free knot (with one unicrusal component).

To do this, we shall use Euler's theorem that for every connected
graph with all vertices of even valency there exists a circuit which
passes once through every edge. Let us choose this circuit to be the
unicursal circuit for $\Gamma'_{n}$ thus defining the framing at
each vertices (two consequent edges at every vertex are decreed to
be formally opposite).

Consider the chord diagram of $\Gamma'_{n}$. This diagram might well
have even and odd chords. Our goal is to construct the chord diagram
of $\Gamma_{n}$ by adding some chords to $\Gamma'_{n}$. Namely, for
every chord $l$ of $\Gamma_{n}$ we shall either do nothing, or add
one small chord at one end of $l$ (linked only with $l$) or add two
small chords at both ends of $l$. Our goal is to show that the
obtain an irreducibly odd chord diagram such that the framed
$4$-graph of $\Gamma'_{n}$ is obtained from the framed $4$-graph of
$\Gamma_{n}$ by smoothing of some vertices.

Note that whenever a chord diagram $Y$ is obtained from a chord
diagram $X$ by adding one chord linked precisely with one chord of
$X$ then the corresponding $4$-graph of can be obtained from the
framed $4$-graph of $X$ by smoothing of some vertices.

Indeed, view Fig. \ref{addcrd}.

\begin{figure}
\centering\includegraphics[width=120pt]{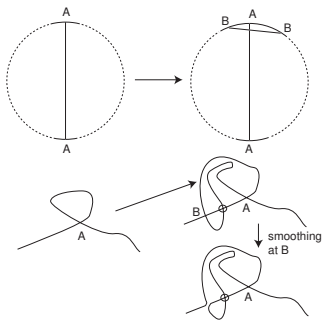} \caption{Addition
of a chord and the inverse operation} \label{addcrd}
\end{figure}

Without loss of generality we may assume that the chord diagram for
$\Gamma'_{n}$ has no {\em solitary} chords (chords not linked with
any other chord).

Now, to every {\em odd} chord of $\Gamma'_{n}$ we add one small
chord on one end of it. To every even chord of $\Gamma'_{n}$ we add
two chords on both flanks. This will guarantee that the resulting
chord diagram is odd (all small chords are odd since each of them is
linked with exactly one chord). Besides, this guarantees that the
resulting chord diagram (or framed $4$-graph) is irreducible.

We shall distinguish between {\em former} chords (belonging to
$\Gamma'_{n}$ and {\em new} chords (small added chords).

Now, no two former chords (for $\Gamma'_{n}$) can be operated on by
a second decreasing Reidemeister move: for each two chords of such
sort $a,b$ there is at least one chord $c$ distinct from $a,b$ which
is linked with $a$ and not with $b$ (it suffices to take one of the
two small chords linked with $a$). A former chord can not
participate in a second Reidemeister move together with a new chord
because every former chord is linked with at least one former chord
and at least one new chord it is linked with, and every new chord is
linked with exactly one chord.

If two new chords $x$ and $y$ are linked with different former
chords, they can not participate in the second Reidemeister move;
neither they can if they are linked with the same former chord: in
this case, since the former chord (say, $z$) is not solitary, there
is at least one chord $w$ lying in between $x$ and $y$, so, the
endpoints of $x$ and $y$ can not be adjacent.

Now, an obvious estimate shows that the number of chords of
$\Gamma_{n}$ does not exceed $3n$.

\end{proof}

\begin{re}
In this direction, one can prove a bit more than stated in the main
theorem: the number of virtual crossings grows  quadratically with
respect to the number of classical crossings not only for virtual
knots, but also for virtual knots considered modulo virtualisation.
\end{re}

\end{document}